\documentclass[a4paper,12pt]{article}
\usepackage{amsmath,amsfonts,enumerate,amssymb,amsthm,color}
\usepackage{tikz, tikzscale}
\usepackage{ifthen}
\usetikzlibrary{arrows, calc,matrix, decorations.pathmorphing}
\title{Which Haar graphs are Cayley graphs?}
\author{Istv\'an Est\'elyi,$^{a,b}$   \; Toma\v{z} Pisanski$^{c,a,b}$ \\  [+0.75ex]
$^a$ {\small FMF, University of Ljubljana, Jadranska 19, 1000 Ljubljana, Slovenia} \\ [-0.5ex]
$^b$ {\small IAM, University of Primorska, Muzejski trg 2, 6000 Koper, Slovenia}\\[-0.5ex] 
$^c$ {\small FAMNIT, University of Primorska, Glagolja\v{s}ka 8, 6000 Koper, Slovenia} \\ [-0.5ex]
}
\date{}

\newtheorem{thm}{Theorem}
\newtheorem{lem}[thm]{Lemma}
\newtheorem{cor}[thm]{Corollary}
\newtheorem{prop}[thm]{Proposition}
\newtheorem{prob}{Problem}
\newtheorem{defi}{Definition}
\newtheorem{rem}{Remark}

\def\Z{\mathbb{Z}}
\DeclareMathOperator{\cay}{Cay}

\DeclareMathOperator{\aut}{Aut}

\newcommand{\comment}[1]{}

\makeatletter
\providecommand*{\xmapstofill@}{%
  \arrowfill@{\mapstochar\relbar}\relbar\rightarrow
}
\providecommand*{\xmapsto}[2][]{%
  \ext@arrow 0395\xmapstofill@{#1}{#2}%
}
\makeatother

\begin{document}

\maketitle

\let\thefootnote\relax\footnote{
The first author was supported by the Young Researcher grant of ARRS (Agencija za raziskovanje Republike Slovenija),
and the ARRS grant no.\ P1-0294. The second author was supported  by the ARRS grant no.\ P1-0294, research project no.\  J1-6720 and by ESF grant Eurocores Eurogiga--GReGAS.
\\  [+0.5ex]
{\it  E-mail addresses:} istvan.estelyi@student.fmf.uni-lj.si (Istv\'an Est\'elyi),  tomaz.pisanski@upr.si (Toma\v{z} Pisanski).
}

\begin{abstract}
For a finite group $G$ and subset $S$ of $G,$ the Haar graph $H(G,S)$ is a bipartite regular graph, defined as a regular $G$-cover of a dipole with $|S|$ parallel arcs labelled by elements of $S$. 
If $G$ is an abelian group, then $H(G,S)$ is well-known to be a Cayley graph; 
however, there are examples of non-abelian groups $G$ and subsets $S$ when this is not the case. In this paper we address the problem of classifying finite non-abelian groups $G$ with the property that every Haar graph $H(G,S)$ is a Cayley graph.  An equivalent condition for $H(G,S)$ to be a Cayley graph of a group containing $G$ is derived 
in terms of $G, S$ and $\aut G$.  It is also shown that the 
dihedral groups, which are solutions  to the above problem, are $\Z_2^2,D_3,D_4$ and 
$D_{5}$.
\medskip

 \medskip

\end{abstract}

\section{Introduction}

All graphs in this paper will be finite and undirected and all groups will be finite.  Recall that, given a group $G$ and a subset $S$ of $G$ with $1_G \notin S$ and 
$S=S^{-1},$ the \emph{Cayley graph} $\cay(G,S)$ is the graph with vertex set $G$ and edges of the form $[g,sg]$ for all $g \in G$ and $s \in S$. A natural generalization of Cayley graphs are the so called Haar graphs introduced by Hladnik et al.\ \cite{HlaMP02} as follows. 
Given a group $G$ and an arbitrary subset $S$ of $G,$ the \emph{Haar graph} $H(G,S)$ is  the voltage graph of a dipole with no loops and $|S|$ parallel edges (from the \emph{white} to the \emph{black} vertex), labeled 
by elements of $S$. More formally, the vertex set of $H(G,S)$ is  
$G \times \{0,1\},$ and the edges are of the form 
$[(g,0),(sg,1)], \; g \in G, s \in S.$
If it is not ambiguous, instead of $[(x,0),(y,1)]\in E(H(G,S))$ we will rather use the notation $(x,0) \sim (y,1)$. The name \emph{Haar graph} comes from the fact that, when $G$ 
is an abelian group the Schur norm of the corresponding adjacency matrix can be easily evaluated via the so-called Haar integral on $G$ (see \cite{Hla99}).  In a more general setting a dipole with the same number of loops and semi-edges on both poles gives rise to a covering graph which is sometimes referred to as a \emph{bi-Cayley graph} of $G$. The generalized Petersen graphs and rose window graphs \cite{W08, KKM10} form notable subfamilies of bi-Cayley graphs (see Fig.~1). For more information on voltage graphs, we refer to \cite{GroT87}. 



\begin{figure}
\begin{center}
\includegraphics[]{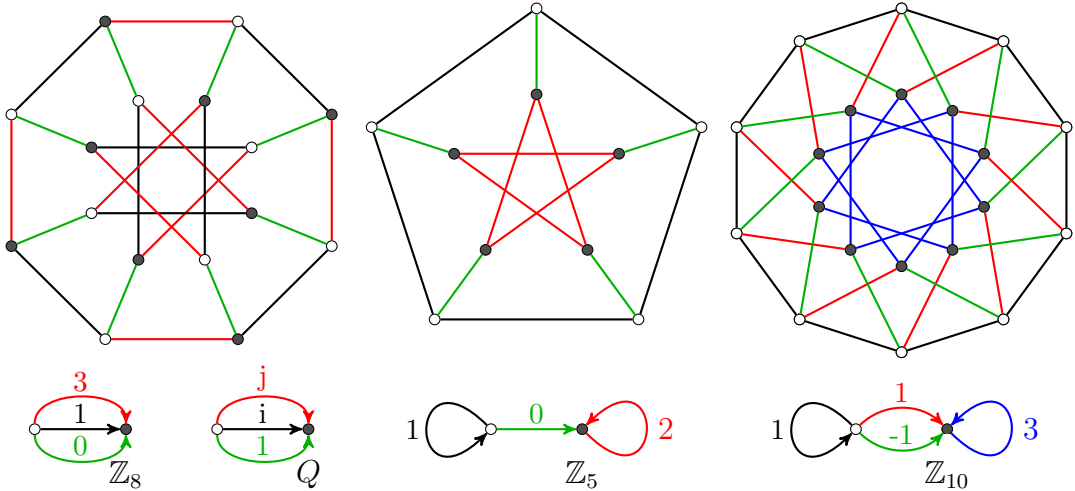}
\end{center}
\centering
\caption{(Left) The M\"{o}bius-Kantor graph as a Haar graph over the cyclic group $\mathbb{Z}_8$ with voltages $S=\{0,1,3\}$ and over the quaternion group $Q$ with voltages $S=\{1,i,j\}$. (Middle) The Petersen graph is a bi-Cayley graph over the cyclic group $\mathbb{Z}_5$, but it is not a Haar graph. (Right) The rose window graph $R_{10}(2,3)$ is a 4-valent bi-Cayley graph.}
\end{figure}

Lately, symmetries of Haar graphs and bi-Cayley graphs have been studied by several authors, cf.\  \cite{DuX00, ExoJ11, HlaMP02, JinL10,KoiK14,KKP14,KKM10, KovMMM09,Lu03,LuWX06,Pis07,W08,ZhouF14}. However, the terminology is not completely unified. To avoid ambiguity, we refer to the specific results of other papers by using the terminology of the present paper. 

Koike and co-authors \cite{KoiK14,KKP14} have studied certain cases where some unexpected automorphism may appear in cyclic Haar graphs.   

Feng and Zhou \cite{ZhouF14} have constructed an infinite family of non-Cayley vertex transitive cubic graphs as bi-Cayley graphs of abelian groups. 

Exoo and Jajcay \cite{ExoJ11} constructed small graphs of large girth, i.e., approximate cages, as Haar graphs. In particular, the (3,30)-cage they constructed as a Haar graph over $SL(2,83)$ with voltages 
$$
  \begin{bmatrix}
    1 & 0  \\
    0 & 1 
  \end{bmatrix},
  \begin{bmatrix}
    0 & 1  \\
    -1 & 6 
  \end{bmatrix},
  \begin{bmatrix}
    1 & 11  \\
    23 & 5
  \end{bmatrix},
$$
 is still the smallest known. It is perhaps of interest to note that the voltage graph and covering graph techniques they use have been used already in the eighties to produce a restricted family of trivalent
graphs of arbitrarily large girth \cite{ShaP82}.  
 
 The concept of Cayley- and bi-Cayley graphs can be further generalized by assuming a higher number of orbits of a semiregular group of automorphisms, e.g., Kutnar et al. \cite{KMMS09} characrterized strongly regular tri-Cayley graphs. 
  
Instead of following this path, we restrict our attention to Haar graphs. Namely, we are interested in determining the exact relationship between Haar graphs and Cayley graphs, which is motivated by their common origins. We also believe it might help better understand the usefulness and limitations of these constructions. Hladnik et al. \cite{HlaMP02} showed that every Haar graph of a cyclic group is a Cayley graph.
It will become apparent that this property can easily be generalized for Haar graphs of 
any abelian group (see  Lemma~\ref{basics}(iii)). On the other hand, Lu et al.\ \cite{LuWX06} have constructed cubic semi-symmetric graphs, i.e., edge- but not vertex-transitive graphs, as Haar graphs of alternating groups. Clearly, as these graphs are not vertex-transitive, they are examples of Haar graphs which are not Cayley graphs.

It is natural to ask which non-abelian groups admit at least one Haar graph 
that is not a Cayley graph, or putting it another way, we pose here the following problem:

\begin{prob}\label{prob1}
Determine the finite non-abelian groups $G$ for which all Haar graphs $H(G,S)$ are  
Cayley graphs.
\end{prob}

Let $X=H(G,S)$, where $G$ is an arbitrary finite group and $S$ be an arbitrary subset 
of $G$.  For $g \in G,$ let $g_R$ be the permutation of $G \times \{0,1\}$ 
defined by $(x,i)^{g_R}=(x g,i),$ and let $G_R = \{ g_R : g \in G\}$. 
It immediately follows that $G_R \le \aut X,$ where $\aut X$ denotes the full automorphism group of $X$.  By the well-known result of Sabidussi, $X$ is a Cayley graph exactly when $\aut X$ contains a subgroup $A$ acting regularly on the vertex set, and in this case we also say 
that $X$ is a Cayley graph of the group $A$. In this paper the primary focus is on the special case when the Haar graph $X$ is a Cayley graph of a group containing $G_R$. 
For this purpose we set the following definition: 

\begin{defi}\label{algcay}
A Haar graph $X=H(G,S)$ is \emph{algebraically Cayley} if $G_R \le A$ for some 
subgroup $A \le \aut X$ acting regularly on the vertex set $V(X)$. 
\end{defi}

In the next section we study algebraically Cayley Haar graphs in details. 
The main result will be an equivalent condition for $H(G,S)$ to be algebraically Cayley in terms of $G, S$ and $\aut G$ (Theorem~\ref{haarcayley1}).  
Obviously, the algebraically Cayley property of a Haar graph implies that it is a 
Cayley graph; the converse implication, however, does not hold in general (see Proposition~\ref{nonsplit}).  Nonetheless, it will be shown that every Cayley Haar graph is isomorphic to an algebraically Cayley Haar graph (see Corollary~\ref{cor}). 
This indicates that the algebraically Cayley Haar graphs play a crucial role towards the solution of Problem~\ref{prob1}. 

In Section~3 we turn to Problem~\ref{prob1} by which we focus our attention on 
dihedral groups. In this paper we denote by $D_{n}$ the dihedral group of order $2n$ 
where $n \ge 2$. Note that $D_2$ is the Klein group $\Z_2^2$.  
Using also the results of Section~2, we show that the dihedral 
groups, which are solutions to Problem~\ref{prob1}, are $\Z_2^2, D_3, D_4$ and $D_{5}$.

\section{Algebraically Cayley Haar graphs}

We start with a few basic properties of Haar graphs.

\begin{lem}\label{basics}
Let $G$ be a finite group and $S$ be a subset of $G$. 
\begin{enumerate}[(i)]
\item The Haar graph $H(G,S)$  is connected if and only if the set 
$S S^{-1}=\{s t^{-1} : s,t \in S\}$ generates $G$.
\item The Haar graph $H(G,S)\cong H(G,g S^\alpha h)$ for any $g,h \in G$ and 
$\alpha \in \aut G$. 
\item If $G$ is an abelian group, then $H(G,S)$ is a Cayley graph.
\end{enumerate}
\end{lem}

Part (i) was proved in \cite{DuX00} and part (iii) in \cite{Lu03}. 
Part (ii) follows at once from Theorem~\ref{haarcayley1} (see also \cite{LuWX06}).
\medskip

In this section we study Haar graphs which are also algebraic Cayley 
(see Definition~\ref{algcay}). The main result of the section is the following theorem, 
which gives an exact algebraic condition for a Haar graph to be algebraic Cayley. For an element $g \in G$ we denote by $\iota_g$ the inner automorphism of $G$ 
induced by $g,$ i.e.,  $x^{\iota_g}= g^{-1}x g$ for every $x \in G$.

\begin{thm}\label{haarcayley1}
The Haar graph $H(G,S)$ is algebraically Cayley if and only if $g S^\alpha =S^{-1}$ 
holds for some $g\in G$ and $\alpha\in \aut G,$ where $g^\alpha=g$, and $\alpha^2=\iota_g$.
\end{thm}

\begin{proof} First, suppose some $g\in G$ and $\alpha \in \aut G$ satisfy the above conditions. Consider the mapping 
$$
\sigma: G\times\Z_2\to G\times\Z_2,\ (x,0)\mapsto (x^\alpha,1),\ (x,1)\mapsto(x^{\alpha^{-1}} g,0), 
$$
where $x^{\alpha^{-1}}=x^{(\alpha^{-1})}$. It is straightforward to check that $\sigma$ is bijective. We will show that $\sigma$ is an automorphism of $H(G,S)$ that swaps its partite sets. 

Recall that $(x,0) \sim (y,1) \Leftrightarrow yx^{-1}\in S$, while $(x,0)^\sigma \sim (y,1)^\sigma\Leftrightarrow (x^\alpha,1)\sim (y^{\alpha^{-1}} g,0) \Leftrightarrow x^{\alpha}( y^{\alpha^{-1}} g)^{-1}\in S$. Let us transform the last expression as follows: 
\begin{align*}
x^{\alpha}( y^{\alpha^{-1}} g)^{-1} &=x^\alpha g^{-1}(y^{\alpha^{-1}})^{-1}=x^{\alpha^2\alpha^{-1}} g^{-1}(y^{-1})^{\alpha^{-1}}=(g^{-1}xg)^{\alpha^{-1}}g^{-1}(y^{-1})^{\alpha^{-1}}\\
&=(g^{-1}x^{\alpha^{-1}}g)g^{-1}(y^{-1})^{\alpha^{-1}}=g^{-1}x^{\alpha^{-1}}(y^{-1})^{\alpha^{-1}}=g^{-1}(xy^{-1})^{\alpha^{-1}}.
\end{align*}

One can see that 
$g^{-1}(xy^{-1})^{\alpha^{-1}}\in S\ \Leftrightarrow
g^{-1}xy^{-1}\in S^\alpha\ \Leftrightarrow
 xy^{-1}\in gS^\alpha\ \Leftrightarrow
 yx^{-1}\in S,$ where we used $g S^\alpha=S^{-1}$ in the last step. Thus $\sigma$ is indeed an automorphism of $H(G,S)$. The partite set-swapping property is clear from the definition of $\sigma$. 

Hence the group $K=\langle\sigma, G_R\rangle$ acts transitively on the vertex set of  $H(G,S)$. We are going to show that $K$ is in fact regular. Note that for the above choice of $\sigma$ we have $\sigma^2: (x,0) \mapsto ( x^{\alpha\alpha^{-1}}g,0),\ (x,1)\mapsto( x^{\alpha^{-1}\alpha}g^\alpha,1)$, implying $\sigma^2=g_R$. In order to prove  $G_R \triangleleft K$ we will show that  $\sigma^{-1}h_R\sigma\in G_R$ for every $h\in G$.
\begin{align*}
&(x,0)\xmapsto[]{\sigma^{-1}}(x^\alpha g^{-1},1)\xmapsto[]{h_R}(x^\alpha g^{-1}h,1)\xmapsto{\sigma }(x(g^{-1}h)^{\alpha^{-1}}g,0)\\
&(x,1)\xmapsto[]{\sigma^{-1}}(x^{\alpha^{-1}} ,0)\xmapsto[]{h_R}(x^{\alpha^{-1}} h,0)\xmapsto{\sigma }(xh^\alpha,1).
\end{align*}

As one can see $(g^{-1}h)^{\alpha^{-1}}g=g^{-1}h^{\alpha^{-1}}g=h^{\alpha^{-1}\alpha^2}=h^\alpha$. 
Hence $\sigma^{-1}h_R\sigma=(h^\alpha)_R\in G_R$. This, together with $\sigma^2=g_R$ already implies that $|K:G_R|=2$. Therefore, $K \leq \aut H(G,S)$ is regular by the Orbit-stabilizer Lemma. It follows from Sabidussi Theorem, that $H(G,S)$ 
is a Cayley graph over $K$.

For the other implication suppose there exists $\sigma\in\aut H(G,S)$ such that $\langle G_R,\sigma\rangle$ is regular and $\sigma:(1_G,0)\mapsto(1_G,1)$. Then $\sigma^2\in G_R$, so there exists $g \in G$ satisfying $\sigma^2=g_R$. Note that $\sigma$ normalizes $G_R$, hence there exists $\alpha \in \aut G,$
well defined by $(h^\alpha)_R = \sigma^{-1} h_R \sigma$ for all $h \in H$.
In particular, $(g^\alpha)_R=\sigma^{-1}g_R\sigma=\sigma^2=g_R,$ and hence
$g^\alpha=g$. Then $(h^{\alpha^2})_R=\sigma^{-2}h_R\sigma^2=g_R^{-1}h_Rg_R=
(g^{-1}hg)_R$ $(\forall h \in G),$ implying $\alpha^2=\iota_g$.
 Also, $h_R \sigma=\sigma (h^\alpha)_R$ for all $h \in G$.

The neighbourhood of $(1_G,0)$ is $(S,1)$, thus the neighbourhood of $(1_G,0)^\sigma$ is $(S,1)^\sigma$, but since $(1_G,0)^\sigma=(1_G,1)$, we obtain that $(S,1)^\sigma=(S^{-1},0)$. Let $s\in S$. Then $(s,1)^\sigma=(1_G,1)^{s_R\sigma}=(1_G,1)^{\sigma (s^\alpha)_R}=(1_G,0)^{\sigma^2 (s^\alpha)_R}=(gs^\alpha,0)$. This shows that $(S,1)^\sigma=(gS^{\alpha},0),$ and we have obtained $S^{-1}=gS^\alpha$. 
This completes the proof of the theorem.
\end{proof}

\begin{cor}\label{hcomplement}The Haar graph $H(G,S)$ is algebraically Cayley if and only if $H(G,G\setminus S)$ is algebraically Cayley.
\end{cor}

\begin{proof}
By symmetry it suffices to prove one implication only. Suppose $H(G,S)$ is algebraically Cayley. Let $g\in G$ and $\alpha\in \aut G$ be the elements given by Theorem \ref{haarcayley1}. Since the maps $x\mapsto gx^\alpha$ and $x\mapsto x^{-1}$ are $G\to G$ bijections, we have $g(G\setminus S)^\alpha=G\setminus gS^\alpha=G\setminus S^{-1}=(G\setminus S)^{-1}$. Hence $H(G,G\setminus S)$ is algebraically Cayley.
\end{proof}

Obviously, the algebraically Cayley property of a Haar graph implies that it is a Cayley graph. To see that the converse implication does not hold in general we have the following example.

\begin{prop}\label{nonsplit}
Let $G$ be a nonsplit metacyclic $p$-group for an odd prime $p$, $N$ be a subgroup 
of $G$ of index $p,$ and $x$ be any element in $G \setminus N$. Then the Haar graph 
$H(G,N \cup \{x\})$ is a Cayley graph which is not algebraically Cayley. 
\end{prop}

\begin{proof}
Let $X=H(G,N \cup \{x\})$ and let $p^n$ be the order of $G$. 

We first prove that $X$ is a Cayley graph. 
Consider the $2p$ $N_R$-orbits on the vertex set of $X$. These orbits are joined by the edges of $X$ into a $2p$-cycle such that the subgraph induced by two consecutive 
orbits are either a complete bipartite graph $K_{p^{n-1},p^{n-1}},$ or a perfect matching $p^{n-1}K_2$. Moreover, the complete bipartite graphs and the perfect matchings alternate (see Fig.~2). This allows us to relabel the vertices of $X$ 
by the group $\Z_{2p} \times \Z_{p^{n-1}}$ such that two vertices 
$(i_1,j_1)$ and $(i_2,j_2)$ are connected if and only if:
$$
i_2-i_1 = 1 \text{ and } (i_1 \in \{0,2,\dots,2p-2\} \text{ or } j_1=j_2).
$$

\begin{figure}[!htb]
\begin{center}
\includegraphics[]{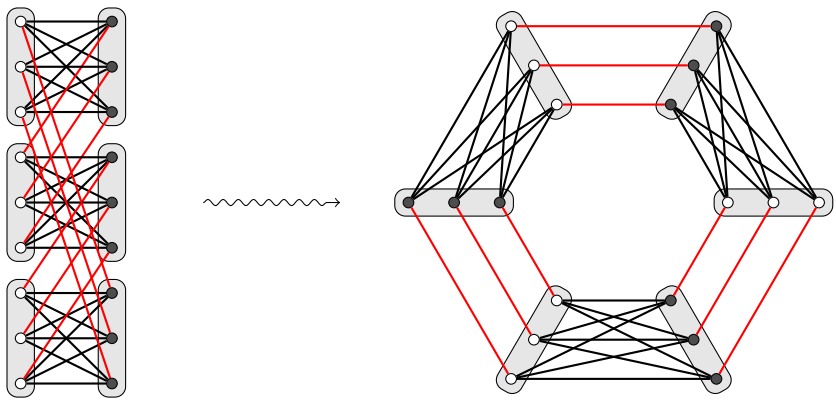}
\end{center}
\centering
\caption{The circular rearrangement of $N_R$-orbits unveils the symmetries $\alpha,\beta,\gamma$.}
\end{figure}

Define the permutations $\alpha, \beta$ and $\gamma$ of $\Z_{2p} \times \Z_{p^{n-1}}$ as for every $(i,j) \in \Z_{2p} \times \Z_{p^{n-1}},$
\begin{eqnarray*}
\alpha     &:& (i,j) \mapsto (i,j+1) \\
\beta       &:& (i,j) \mapsto (i+2,j) \\ 
\gamma  &:& (i,j) \mapsto (-i+1,j),
\end{eqnarray*}
where all operations are carried out in $\Z_{2p} \times \Z_{p^{n-1}}$.
It is straightforward to verify that each of $\alpha, \beta$ and $\gamma$ is 
an automorphism of $X,$  $\alpha$ commutes with both $\beta$ and $\gamma,$ 
and $\gamma^{-1} \beta \gamma=\beta^{-1}$.  All these imply that $X$ is a 
Cayley graph of the group $\big\langle \alpha,\beta,\gamma \big\rangle,$ the latter 
group being isomorphic to $\Z_{p^{n-1}} \times D_{p}$.  

Next we prove that $X$ is not algebraically Cayley. 
The order of a nonsplit metacyclic $p$-group 
was determined by Menegazzo \cite{Men93} for an odd prime $p;$ 
in particular, $\aut G$ is also a $p$-group.  This will be crucial in the argument below. 
By Theorem~\ref{haarcayley1}, assume, towards a contradiction, that 
there exist $y \in G$ and $\sigma \in \aut G$ such that $y^\sigma=y,$  
$\sigma^2=\iota_y,$ the inner automorphism of $G$ induced by  $y,$ and 
\begin{equation}\label{N}
y N^\sigma \cup \{ y x^\sigma \} = N \cup \{x^{-1}\}.
\end{equation}
It follows at once that $y=x^{-1}$ or $y \in N$.

Suppose that $y=x^{-1}$. The \eqref{N} reduces to $N^\sigma \cup \{x\} = xN \cup \{1\}$. This implies that, for any $z \in N, z \ne 1,$ $x z, x z^2$ are in $N^\sigma,$ 
and thus so is $z = (x z)^{-1} x z^2$. We obtain that $N^\sigma = N$. But then 
$N \cap x N \ne   \emptyset,$ a contradiction. 

Suppose that $y \in N$. Then $y N^\sigma= (y N)^\sigma = N^\sigma,$ and 
\eqref{N} shows that $|N^\sigma \cap N| \ge |N|-1,$ implying that $N^\sigma=N$. 
This means that $\sigma$ induces an automorphism of the factor group $G/N,$ let 
us denote this automorphism by $\bar{\sigma}$.  
Since $\aut G$ is a $p$-group, $\bar{\sigma}$ is of $p$-power order. 
Using this and that $G/N \cong \Z_p,$ we get that $\bar{\sigma}$ is the identity 
mapping. Therefore, $(N x)^{\bar{\sigma}}= N x,$ and so 
$y x^\sigma \in N x$. On the other hand, from \eqref{N} and since $N^\sigma=N,$
$y x^\sigma = x^{-1}$ follows. These imply that $x^{-1} \in N x,$ and so 
$x^2 \in N,$ which is a contradiction.  This completes the proof of the proposition.
\end{proof}

In our next proposition we generalize \cite[Proposition~4.6]{HlaMP02}.

\begin{prop}\label{cayhaarbip}
A Cayley graph $\cay(G,S)$ is a Haar graph if and only if it is bipartite. 
\end{prop}

\begin{proof} The only if part is trivial. For the if part, let 
$X=\cay(G,S),$ $\aut X$ be the full automorphism group of $X,$ and let 
$(\aut X)^+$ be the group of automorphisms of $X$ that fix the bipartition classes. 

We suppose first that $X$ is a connected graph. In this case one can see that 
$|\aut X: (\aut X)^+ |=2$. By some abuse of notation we also denote by $G$  
the group consisting of the permutations in the form $x \mapsto x g,$ where $g$ 
runs over the set $G$. Then $G \not\le (\aut X)^+$, hence the product 
$G (\aut X)^+ = \aut X$. Now, 
\begin{align*}
\frac{|G| |(\aut X)^+|}{|G\cap(\aut X)^+|} &=|\aut X|,\\
|G:G\cap(\aut X)^+| &=|\aut X:(\aut X)^+|=2.
\end{align*}
Let $G^+=G\cap(\aut X)^+$. Since $G$ is regular, $G^+$ acts  semiregularly with the two orbits being the partite sets of $X$. Hence $X$ is a Haar graph over $G^+$. This settles the proposition for connected Cayley graphs.

Suppose next that $X$ is disconnected. Equivalently, $K$, the group  
generated by $S$ is a proper subgroup of $G$. Furthermore, $X$ consists of 
$|G:K|=m$ components, denoted by $X_0,\dots, X_{m-1}$, all isomorphic to $\cay(K,S)$. Let $r_0,r_2,\dots,r_{m-1}$ be a complete set of representatives of right cosets of $K$. 
It will be convenient to regard the indices of $r_i$'s as elements of $\Z_m$.
The vertex set $V(X_i) = K r_i$ for $i \in \Z_m$.
The Cayley graph $\cay(K,S)$ is connected and bipartite, and therefore, 
there exists a subgroup $K^+$ of $K$ (here $K$ is a 
permutation group acting on itself), which is semiregular, and the orbits of which are the partite sets of $\cay(K,S)$.  
Now, define the action of $K^+ \times \Z_m$ on 
$G$ by letting 
$$
g^{(k,j)} = (x^k)r_{i+j}  \stackrel{\rm def}{\iff} g=x r_i,  x \in K, i \in \Z_m,
$$
where the sum $i+j$ is from $\Z_m$. 
Observe that the above image of $g$ under $(k,j)$ is 
well-defined, because $g$ decomposes uniquely as $g= x r_i$ with $x \in K$ and $i \in \Z_m$. It can be easily checked that $K^+ \times \Z_m$ acts semiregularly on 
$G$ with two orbits. Denote also by $K^+ \times \Z_m$ the obtained permutation 
group of $G$. Let $[h,g]$ be an edge of $X$. Then  $h=xr_i$ and $g=sx r_i$ for 
some $x \in K,$ $i \in \Z_m$ and $s \in S$. Using that $S \subset K,$ we get that $[h^{(k,j)},g^{(k,j)}]$ is also an edge of $G$ for every 
$(k,j) \in K^+ \times \Z_m,$ and thus $K^+ \times \Z_m \le \aut X$. 
The subgroup $K^+$ have $2m$ orbits and these are exactly 
the partite sets of the components $X_0,\dots,X_{m-1}$. 
The group $K^+ \times \Z_m$ permutes these orbits  in two $m$-cycles, implying 
that $K^+ \times \Z_m$ has two orbits which are partite sets of $X$. 
Thus $X$ is a Haar graph of $K^+ \times \Z_m$.
\end{proof}

An immediate consequence of Proposition~\ref{cayhaarbip} is that 
every Cayley Haar graph can be interpreted  as an algebraically Cayley Haar graph. 
The precise statement is given below.

\begin{cor}\label{cor}
If a Haar graph $H(G,S)$ is a Cayley graph, then there exist a group $\tilde{G}$ and subset $\tilde{S}$ of $\tilde{G}$ such that $H(G,S) \cong H(\tilde{G},\tilde{S})$, and $H(\tilde{G},\tilde{S})$ is algebraically Cayley. 
\end{cor}

We finish the section proving a property of groups all of whose  
Haar graphs are algebraically Cayley. 

\begin{prop}
Let $G$ be a finite group all of whose Haar graphs are 
algebraically Cayley. 
\begin{enumerate}
\item[(i)] If $H$ is a subgroup of $G$, then every Haar graph of $H$ is algebraically Cayley.
\item[(ii)] If $N$ is a characteristic subgroup of $G$, then every Haar graph $H(G/N,\mathcal{S})$ is algebraically Cayley.
\end{enumerate} 

\end{prop}

\begin{proof} \textit{(i)} Let $S$ be a subset of $H$ containing $1$. We will show that there exist $h\in H$ and $\beta\in\aut H$ such that $h^{\beta}=h$, $\beta^2=\iota_h$ and $hS^{\beta}=S^{-1}$. Because of Corollary~\ref{hcomplement}, it suffices to prove the existence of such $h$ and $\beta$ for $|S|\geq |H|/2$. Choose $t\in G\setminus H$ and apply Theorem~\ref{haarcayley1} for $H(G,tS)$. Thus we get elements $g\in G$ and $\alpha\in \aut G$ such that $g^\alpha=g$, $\alpha^2=\iota_g$ and 
\begin{align*}
gt^\alpha S^\alpha&=S^{-1}t^{-1},\\
gt^\alpha t S^{\alpha\iota_t}&=S^{-1}.
\end{align*}
The LHS of the last equation is contained in the coset $gt^\alpha t H^{\alpha\iota_t}$, while the RHS is a subset of $H$. Therefore, $S^{-1} \subseteq gt^\alpha t H^{\alpha \iota_t} \cap H =
gt^\alpha t(H^{\alpha \iota_t} \cap H),$ because
$gt^\alpha t \in S^{-1} \subseteq H$.
We obtain that $S^{-1}$ is contained in a left coset of the group
$K=H^{\alpha \iota_t}\cap H$. Namely, because of $1\in S$, we have $S\subseteq K$, which implies $|H|/2\leq |S|\leq |K|\leq |H|$. Hence $|H:K|$ is either 1 or 2. If $|H:K|=2$, then $S=K$, in this case one can choose $h=1$ and $\beta=id_H$. If $H=K$, then $\alpha\iota_t\in \aut H$. It
is straightforward to check that in this case $h=gt^\alpha t$ and $\beta=\alpha\iota_t$ satisfies the conditions.

\textit{(ii)} Let $\pi$ be the natural projection $\pi : G \to G/N, g \mapsto Ng$ (recall that the factor group $G/N$ consists of the $N$-cosets of $G$).  Fix a subset $S$ of $G$ such that 
$\mathcal{S} = \{N s : s \in S\}$.  The Haar graph $H(G,NS)$ is algebraic Cayley, 
hence by Theorem~\ref{haarcayley1}, there exist $g \in G$ and 
$\alpha \in \aut G$ such that $g(NS)^\alpha=(NS)^{-1}=NS^{-1}$, $\alpha$ fixes $g$ and $\alpha^2=\iota_g$, the inner automorphism induced by $g$. 

Since $N$ is a characteristic subgroup of $G,$ there is a unique automorphism  
$\beta \in \aut (G/N)$ such that the diagram below  commutes.

\begin{center}
\begin{tikzpicture}
\matrix (m) [matrix of math nodes,row sep=3em,column sep=4em,minimum width=2em]
  {
     G & G \\
     G/N & G/N \\};
  \path[-stealth]
    (m-1-1) edge node [left] {$\pi$} (m-2-1)
            edge [left] node [above] {$\alpha$} (m-1-2)
    (m-2-1.east|-m-2-2) edge 
            node [above] {$\beta$} (m-2-2)
    (m-1-2) edge node [right] {$\pi$} (m-2-2);
\end{tikzpicture}
\end{center}

\noindent 
Equivalently,  $(N x)^\beta=N x^\alpha$ for every $x \in G$. 
But then $N g \mathcal{S}^\beta=\{N g s^\alpha :s \in S\}= \{N s^{-1} : s \in S\}=\mathcal{S}^{-1}$, $(N g)^\beta=N g$ and $\beta^2=\iota_{Ng}$, and hence 
$H(G/N,\mathcal{S})$ is algebraically Cayley.
\end{proof}

\section{Cayley Haar graphs over dihedral groups}

In this section we solve Problem~\ref{prob1} for dihedral groups. 
Recall that, we denote by $D_{n}$ the dihedral group of order $2n$ for $n \ge 2$.
The main result here is the following theorem, which we are going to settle by the end of 
the section.

\begin{thm}\label{dihedral}
Each Haar graph of order $4n$, of the dihedral group $D_{n}$ of order $2n$, is a Cayley graph if and only if 
$n \in \{2,3,4,5\}$. 
\end{thm}

We start with a lemma about generalized dihedral groups. We recall that, for  an 
abelian group $A,$ the \emph{generalized dihedral group} $D(A)$ is the group 
$\langle A,t \rangle,$ where $t$ is an involution not contained in $A$ and 
$x^t=x^{-1}$ for every $x \in A$ (cf. \cite[page 215]{Sco64}). 

\begin{lem}\label{5valenthaar}
Let $A$ be a finite abelian group. Then every Haar graph of $D(A)$ with valency 
at most $5$ is a Cayley graph.
\end{lem}

\begin{proof}
Suppose we are given $H(D(A),S)$ with $|S| \leq 5$. 
Let $S_1 = S \cap A,$ and $S_2$ be the subset of $A$ that satisfies 
$S_2 t  = S \cap A t$.  At least one of $S_1$ and $S_2$ has at most $2$ elements, 
and by Lemma~\ref{basics}(ii), we may assume that $|S_1| \leq 2$ and $1 \in S_1$, where $1$ denotes the identity element of $D(A)$.  
In view of  Theorem~\ref{haarcayley1} it is sufficient to show the following: 
\begin{equation}\label{condition}
g S^\alpha = S^{-1} \text{ for some } g \in D(A) \text{ and } \alpha \in \aut D(A),  
\end{equation}
where $g^\alpha=g$ and $\alpha^2=\iota_g,$ the latter being the inner automorphism of $D(A)$ 
induced by $g$. 

If $S_1=\{1\},$ then $S=S^{-1}$. Thus \eqref{condition} holds by letting 
$g=1$ and $\alpha$ be the identity mapping. 

Let $S_1=\{1,a\}$. Let $g=a^{-1}$, and let $\alpha\in \aut D(A)$ be the automorphism of $D(A)$ defined by $x^\alpha=x$ for every $x \in A$ and $t^\alpha=a t$. Then 
$g S^\alpha=S^{-1}, g^\alpha=g, \alpha^2=\iota_{a^{-1}}$, and \eqref{condition} 
holds also in this case. The lemma is proved.
\end{proof}

\begin{rem}\label{val6}
Lemma~\ref{5valenthaar} is sharp in the sense that $5$ is the highest valency for which the statement holds. 
Using {\sc Sage} we determined the smallest dihedral Haar graph that is not a Cayley graph. It is not even vertex-transitive. It is isomorphic to
$$H(D_{6},\{1,a,a^3,b,ab,a^3b\}).$$
It seems to be the first element in an infinite series
of such graphs.  Our computations show that graphs $H(D_{n},\{1,a,a^3,b,ab,a^3b\}), 6 \leq n \leq 100$
are not vertex-transitive.
\end{rem}

\begin{lem}\label{small} 
All Haar graphs of the dihedral groups $D_3$, $D_4$ and $D_{5}$ are 
algebraically Cayley. 
\end{lem}

\begin{proof}
Let $n\in \{3,4,5\}$ and let us consider the Haar graph $H(D_{n}, S)$. Then $|D_{n}|\leq 10$, implying that either $|S|\leq 5$ or $|D_{n}\setminus S|\leq 5$.
If $|S|\leq 5$, then $H(D_{n}, S)$ is algebraically Cayley by Lemma~\ref{5valenthaar}.
If $|S|> 5$, then $H(D_{n}, D_{n}\setminus S)$ is algebraically Cayley by Lemma~\ref{5valenthaar}. Hence $H(D_{n}, S)$ is algebraically Cayley by Corollary~\ref{hcomplement}.
\end{proof}

In view of the examples in Remark~\ref{val6} and of Lemma~\ref{small}, 
Theorem~\ref{dihedral} follows from the next proposition:   

\begin{prop}\label{dihhaarautmin}
Let $n > 7,$ $D_{n}=\langle a,b  \mid  a^n=b^2=1, bab=a^{-1} \rangle$ and let $S=\{1,a,a^3,b,ab,a^2b,a^4b\}$. Then $\aut H(D_{n},S)\cong D_{n}$.  
\end{prop}

\begin{proof}
We have checked the cases $n\leq 24$ using {\sc Magma} \cite{BosCP97}. Now, suppose $n>24$ 
and let $X=H(D_{n},S)$, let $V=D_{n} \times \{0,1\}$ and $V_i=D_{n}\times \{i\}$ where $i \in \{0,1\}$.
Let $(\aut X)^+$ be the setwise stabilizer of $V_0$ in $\aut X$. 
Fix $r \in \{0,1\}$. Define the action of the group $(\aut X)^+$ on $D_{n}$ by letting 
$$
x^\gamma = y \text{ if and only if } (x,r)^\gamma=(y,r) \;  
(x \in D_{n}, \, \gamma \in (\aut X)^+).
$$

Let $G$ be the permutation group of $D_{n}$ induced by the above action.
Note that $(D_{n})_R \le G,$ since the right regular action of $D_{n}$ fixes
the bipartition classes, where by some abuse of notation we shall denote also by $(D_{n})_R$ the group of all permutations $x \mapsto xd$ ($x\in D_{n}$)\medskip
				
\noindent{\bf Claim.} $G=(D_{n})_R$.
\medskip

For a positive integer $k,$ let 
$$
S_k = \{ d \in D_{n}: d=x y^{-1} \text{ for exactly $k$ pairs } (x,y) \in S \times S\}.
$$

In our case the following sets are obtained via direct computation: 
$S_1=\{a^4,a^{-4}\},$ $S_2=\{a^3,a^{-3},b,a^7 b\},$ 
$S_3=\{a,a^{-1},a^2,a^{-2}\},$ $S_4=\{ab,a^2b,a^3b,a^4b,a^5b\},$ 
$S_7=\{1\},$ and $S_k = \emptyset$ if $k \notin \{1,2,3,4,7\}$.
It can be easily seen that $G \le \aut \cay(D_{n},S_k)$ for every $k \in \{1,2,3,4\}$. 
Now let $C=\langle a\rangle$. Since $\langle S_3 \rangle =C,$ $G$ preserves the partition of 
$D_{n}$ into $C$ and $D_{n} \setminus C$.
Let $G_1$ be the stabilizer of $1$ in $G$. 
Then $G_1$ leaves $C$ setwise fixed, and applying this to $\aut \cay(D,S_2)$ 
gives us that $G \le \aut \cay(D_{n},\{a^3,a^{-3}\})$ and 
$G \le \aut \cay(D_{n},\{b,a^7b \})$ also hold.  

Choose an element $\gamma \in G_1$. 
Then $a^3$ and $a^{-3}$ are either fixed or switched by $\gamma$.
Suppose that $(a^3)^\gamma = a^3$. 
The component of $\cay(D_{n},\{a^3,a^{-3}\})$ containing the vertex $1$ is 
a cycle, and thus we find that $\gamma$ fixes every element of the subgroup 
$\langle a^3 \rangle$. Since $n > 24,$  
$|\langle a^3 \rangle \cap \langle a^4 \rangle| > 2$. 
Using this and that the component of $\cay(D_{n},\{a^4,a^{-4}\})$ containing $1$ is 
a cycle, $\gamma$ also fixes every element of 
$\langle a^4 \rangle$. Also, since for $i =1,2,$ 
$|\langle a^3 \rangle a^i \cap \langle a^4 \rangle| = 
|\langle a^3 \rangle  \cap \langle a^4 \rangle| > 2,$ we get eventually  
that $\gamma$ fixes every element of $C$. It follows from this and 
$\gamma \in \aut \cay(D_{n},\{b,a^7b\})$ that $\gamma$ fixes also 
every element in $D_{n} \setminus C,$ i.e., $\gamma=1$. It can be derived in the same 
manner that the condition $(a^4)^\gamma = a^4$ forces that $\gamma=1$. 
Therefore, we have shown that $G_1$ acts faithfully on the set 
$T=\{a^3,a^{-3},a^4,a^{-4}\},$ moreover, the permutation group of $T$ induced 
by $G_1$ is contained in the group $\langle (a^3,a^{-3})(a^4,a^{-4}) \rangle$. 

Thus if $G_1 \ne 1,$ then it is generated by an involution, say $\gamma$. 
Also, $|G : (D_{n})_R|=2,$ hence $\gamma$ normalizes $(D_{n})_R,$ and thus 
$\gamma \in \aut D_{n}$  such that $(a^3)^\gamma=a^{-3}$ and  
$(a^4)^\gamma=a^{-4}$. It follows that $a^\gamma = a^{-1}$. 
Then we can write $\{b, a^7b\} = \{b^\gamma, (a^7b)^\gamma \}=
\{b^\gamma,  a^{-7}b^\gamma\}$. Thus either 
$b^\gamma=b$ and $(a^7b)^\gamma= a^7b,$ from which $a^{14}=1,$ 
a contradiction, or $b^\gamma=a^7b$ and $(a^7b)^\gamma= b$. However, in this latter case using $\gamma\in\aut\cay(D_{n},S_4)$ and $\gamma\in G_1$, we get that  $(ab)^\gamma\in S_4\Leftrightarrow a^6\in\{a,a^2,a^3,a^4,a^5\}$, a contradiction because of $n>24$. Thus $G_1$ is trivial, and the claim follows.
\medskip

\noindent{\bf Claim.} $|(\aut X)^+|=2n$.
\medskip

Clearly, $|(\aut X)^+| \ge 2n,$ and equality holds exactly when the stabilizer of 
the vertex $(1,0) \in V$ in $(\aut X)^+$ is trivial. 
Let $\gamma \in (\aut X)^+$ with $(1,0)^\gamma = (1,0)$. 
By the previous claim $\gamma$ fixes the set $V_0$ pointwise, and it acts on 
$V_1$ as a permutation of the form $(x,1) \mapsto (xd,1)$ for some $d \in D_{n}$. 
This implies that  $S d = S,$ hence $S$ is a union of left cosets of the subgroup 
$\langle d \rangle$. It follows easily that $d=1,$ and so $\gamma=1$.
\medskip

\noindent{\bf Claim.} $\aut X=(\aut X)^+$.
\medskip

The index $|\aut X: (\aut X)^+| \le 2,$ hence if 
$\aut X \ne (\aut X)^+,$ then $\aut X$ is regular on $V$. 
In that case Theorem~\ref{haarcayley1} implies $S^{-1} =  d S^\sigma$ for some $d \in D_{n}$ and 
$\sigma \in \aut D_{n}$ such that $\sigma^2$ is the inner automorphism of $D_{n}$ 
induced by $d$. Denote by $\bar{\sigma}$ the automorphism of $C$ obtained 
by restricting $\sigma$ to $C$ (recall that $C$ is characteristic in $D_{n}$). 

It is obvious that $d \in S^{-1}$. Also, since 
$|S \cap C|=3$ and $|S \cap (D_{n} \setminus S)|=4,$ it follows that $d \in \{1,a^{-1},a^{-3}\}$.
Thus for every $x \in C,$ $x^{\sigma^2} = d^{-1}xd=x,$ i.e., 
$\bar{\sigma}$ is of order at most $2$. If $d=a^{-1}$, then $\{a^{\bar{\sigma}},
(a^{\bar{\sigma}})^3\}=\{a,a^{-2}\}$, thus either $a^3=a^{-2}$, or $a^{-6}=a$.
 If $d=a^{-3}$, then $\{a^{\bar{\sigma}},
(a^{\bar{\sigma}})^3\}=\{a^2,a^3\}$, thus either $a^9=a^2$, or $a^6=a^3$.
All of these cases can be quickly excluded using that the order of $a$ is $n > 24$. 

For the last remaining case let $d=1$. Then it follows that $x^{\bar{\sigma}}=x^{-1}$ for every $x \in C$ and 
$\{b,ab,a^2b,a^4b\}=\{b^\sigma,(ab)^\sigma,(a^2b)^\sigma,
(a^4b)^\sigma\}$. Suppose $b^\sigma=a^ib$ for some $i\in\{0,1,2,4\}$. Thus $(a^4b)^\sigma=a^{i-4}b\in \{b,ab,a^2b,a^4b\}$, which is only possible if $i=4$, using $n>24$. But then $(ab)^\sigma=a^3b\notin\{b,ab,a^2b,a^4b\},$ which is a contradiction. Hence the claim follows, completing the proof of the proposition. 
\end{proof}

We finish with another open question, placing Haar graphs into the more general class of vertex-trasitive graphs.

\begin{prob}\label{prob2}
Is there a non-abelian group $G$ and a set $S \subset G$ such that the Haar graph $H(G,S)$ is vertex-transitive but
non-Cayley?
\end{prob}

On a related note it might be interesting to mention \textit{quasi-Cayley} graphs, a class of vertex-transitive graphs that properly contains the class of Cayley graphs, shares many characteristics of the Cayley graphs, and is properly contained in the class of vertex-transitive graphs. Quasi-Cayley graphs were defined by Gauyacq \cite{Gau97} as follows:

\begin{defi} A graph $X$ is \textit{quasi-Cayley}, if the exists a regular family $\mathcal{F}$ of automorphisms, i.e., $\mathcal{F}\subseteq \aut X$ such that for all $u,v\in V(X)$ there exists a unique $\sigma\in\mathcal{F}$ such that $u^\sigma=v$.
\end{defi}

\begin{prop}Each vertex-transitive Haar-graph is quasi-Cayley.
\end{prop}

\begin{proof}
Let a Haar graph $X=H(G,S)$ be vertex-transitive. Then there exists $\sigma\in \aut X$ which swaps the partite sets of $X$. It is straightforward to check that $\mathcal{F}=G_R\cup\sigma G_R$ is a regular family of automorphisms of $X$.
\end{proof}

\section*{Acknowledgement}

The authors are greatly indebted to Istv\'an Kov\'acs for the numerous useful suggestions and improvements. Furthermore, we would like to thank Robert Jajcay for bringing relevant examples and additional results to our attention.

\end{document}